\documentclass{amsart}

\usepackage{amsthm}
\usepackage{amssymb}
\usepackage[all]{xy}

\newtheorem{lemma}{Lemma}
\newtheorem{theorem}[lemma]{Theorem}
\newtheorem{corollary}[lemma]{Corollary}

\newtheorem*{assumption}{Assumption}

\newcommand*{\bpd}{\mathop{\mathrm{bpd}}\displaylimits}

\title[Some criteria for balanced projectivity]{On some criteria for the balanced projectivity of modules over integral domains}
\author[J. E. Mac\'{\i}as-D\'{\i}az]{J. E. Mac\'{\i}as-D\'{\i}az} 
\address{Departamento de Matem\'{a}ticas y F\'{\i}sica, Universidad Aut\'{o}noma de Aguascalientes, Avenida Universidad 940, Ciudad Universitaria, Aguascalientes, Ags. 20131, Mexico}
\email{jemacias@correo.uaa.mx}

\subjclass[2000]{Primary 13C10, 13C05; Secondary 16D40, 13F05}
\keywords{Direct limits, ascending chains of modules, complete decomposability, pure submodules, integral domains}
\date{\today}

\begin{document}

\begin{abstract}
Motivated by Hill's criterion of freeness for abelian groups, we investigate conditions under which unions of ascending chains of balanced-projective modules over integral domains are again balanced-projective. Our main result establishes that, in order for a torsion-free module to be balanced-projective, it is sufficient that it be the union of a countable, ascending chain of balanced-projective, pure submodules. The proof reduces to the completely decomposable case, and it hinges on the existence of suitable families of submodules of the links in the chain. A Shelah-Eklof-type result for the balanced projectivity of modules is proved in the way, and a generalization of Auslander's lemma is obtained as a corollary.
\end{abstract}

\maketitle

\section{Introduction}

Beforehand, recall that a subgroup $H$ of a commutative group $G$ is \emph {relatively divisible} if every equation of the form $k x = a$, with $k \in \mathbb {Z}$ and $a \in H$, has a solution in $H$ whenever it has a solution in $G$. Equivalently, $H$ is a \emph {pure} subgroup of $G$, that is, every system of equations of the form
\begin{equation*}
	\sum _{j = 1} ^m k _{i j} x _j = a _i \quad (i = 1 , \dots , n),
\end{equation*}
with each $k _{i j} \in \mathbb {Z}$ and $a _i \in H$, is solvable in $H$ whenever it is solvable in $G$.

In $1934$, L. Pontryagin proved that a countable, torsion-free abelian group is free if and only if every finite rank, pure subgroup is free \cite {Pontryagin}. Equivalently, every properly ascending chain of pure subgroups of finite rank is finite. From the proof of this result, it follows that a torsion-free abelian group $G$ is free if there exists a countable, ascending chain
\begin{equation}
	0 = G _0 < G _1 < \dots < G _n < \dots \quad (n < \omega),
	\label{CountChain}
\end{equation}
consisting of pure subgroups of $G$ whose union is equal to $G$, such that every $G_n$ is free and countable.

In $1970$, P. Hill established that, in order for an abelian group $G$ to be free, it is sufficient that it be the union of an ascending chain \eqref {CountChain} of free, pure subgroups \cite {Hill}. In other words, Hill proved that the condition of countability on the cardinality of the links $G _n$ in Pontryagin's theorem was superfluous. The proof of this theorem hinges on some important facts about commutative groups, one of them being that subgroups of torsion-free abelian groups can be embedded in pure subgroups of the same rank, another being the existence of some suitable collections of subgroups of free abelian groups.

Hill's criterion of freeness has been recently generalized to a criterion of projectivity for modules over integral domains. More concretely, it has been established that unions of countable, ascending chains of projective, pure submodules over domains are again projective (see Theorem 1.4 of Chapter XVI in \cite {Fuchs-Salce2}). This result generalizes a previous version available in the literature, valid for valuation domains and, as in the case of commutative groups, its proof relies again on the condition of purity, and the existence of suitable collections of submodules. 

The concepts of freeness and projectivity of modules find their generalizations in the notions of complete decomposability and balanced projectivity \cite {Fuchs-Salce2}, respectively, so that it is natural to ask whether there exists a further generalization of Hill's theorem that considers these conditions. Indeed, some progress has been achieved in this direction. For instance, K. M. Rangaswamy proved in 1998 one generalization for modules over valuation domains, namely, that modules over such integral domains are completely decomposable if they are unions of countable, ascending chains of completely decomposable, pure submodules \cite {Rangas}.

In the present work, we extend Rangaswamy's theorem to modules over arbitrary domains, and present generalizations to balanced-projective modules of the well-known Auslander's lemma and the Shelah-Eklof criterion of freeness. Throughout our investigation, modules will be defined over an arbitrary integral domain $R$.

\section{Balanced projectivity}

The concepts of balancedness and complete decomposability of modules are generalizations of the corresponding terms for abelian groups. For example, by a \emph {completely decomposable} module we understand a direct sum of rank-one, torsion-free modules. Free modules are clearly completely decomposable.

A submodule $A$ of a module $B$ is \emph {balanced} in $B$ if $B / A$ is torsion-free, and every rank-one, torsion-free module $J$ has the \emph {projective property} with respect to the exact sequence $0 \rightarrow A \rightarrow B \rightarrow B / A \rightarrow 0$, that is, for every homomorphism $\phi : J \rightarrow B / A$, there exists a homomorphism $\psi : J \rightarrow B$ which makes the following diagram commute:
\begin{equation*}
	\xymatrix{
		 & & & J \ar[d]^{\phi}\ar[ld]_\psi & \\
		0 \ar[r] & A \ar[r] & B \ar[r] & B / A \ar[r] & 0}
\end{equation*}
An exact sequence $0 \rightarrow A \rightarrow B \rightarrow C \rightarrow 0$ is \emph {balanced-exact} if $C$ is torsion-free and $A$ is balanced in $B$.

A module is \emph {balanced-projective} if it has the projective property with respect to every balanced-exact sequence. Completely decomposable modules are balanced-projective; moreover, a torsion-free, balanced-projective module is equivalently a direct summand of a completely decomposable module of the same rank (see \cite {Fuchs-Salce2} for properties on balancedness, balanced projectivity and complete decomposability). Obviously, every projective module is balanced-projective, and direct sums of balanced-projective (resp., completely decomposable) modules are likewise balanced-projective (resp., completely decomposable) modules.

The following result is a straightforward criterion for balanced exactness of sequences.

\begin{lemma}
	\label{Lemma:2}
	Consider the following commutative diagram with exact rows, and torsion-free modules $M$ and $M ^\prime$:
	\begin{equation*}
		\xymatrix{
			0 \ar[r] & L \ar[r]\ar[d] & N \ar[r]^\alpha\ar[d]_\mu & M \ar[d]^\tau\ar[r] & 0\\
			0 \ar[r] & L ^\prime \ar[r] & N ^\prime \ar[r]_\beta & M ^\prime \ar[r] & 0}
	\end{equation*}
	If the top row is balanced-exact and there exists a homomorphism $\rho : M ^\prime \rightarrow M$ such that $\tau \rho = 1$, then the bottom row is likewise balanced-exact.
\end{lemma}

\begin{proof}
	Let $\phi$ be a homomorphism from a rank-one, torsion-free module $J$ into $M ^\prime$. There exists a homomorphism $\sigma : J \rightarrow N$ such that $\alpha \sigma = \rho \phi$. Clearly, $\psi = \mu \sigma$ has the property that $\beta \psi = \phi$, whence the conclusion follows.
\end{proof}

Recall that a submodule $N$ of an $R$-module $M$ is \emph {relatively divisible} if $r N = N \cap r M$, for every $r \in R$ \cite {Warfield}. Equivalently, every equation of the form $r x = a$, with $r \in R$ and $a \in N$, has a solution in $N$ whenever it has a solution in $M$. We say that $N$ is \emph {pure} in $M$ if every system of equations 
\begin{equation*}
	\sum _{j = 1} ^m r _{i , j} x _j = a _i \quad (i = 1 , \ldots , n),
\end{equation*}
with each $r _{i , j} \in R$ and $a _i \in N$, is solvable in $N$ whenever it is solvable in $M$. Evidently, purity implies relative divisibility, and they both coincide over Pr\"{u}fer domains \cite {Cartan-Eilenberg}. Moreover, balanced submodules are relatively divisible, and direct summands are relatively divisible, pure and balanced submodules. We refer to \cite {Fuchs-Salce2} for more properties on these conditions.

The following lemma allows to reduce the study of ascending chains of balanced-projective modules to the case of chains of completely decomposable modules.

\begin{lemma}
	\label {Lemma:6}
	Let $M$ be a torsion-free module for which there exists a countable, ascending chain
	\begin{equation}
		\label {Eq:CountChain}
		0 = M _0 < M _1 < \ldots < M _n < \ldots \quad (n < \omega)
	\end{equation}
	of submodules of $M$, such that
	\begin{enumerate}
		\item[(a)] every $M _n$ is a balanced-projective module,
		\item[(b)] every $M _n$ has countable rank,
		\item[(c)] every $M _n$ is pure (resp.,  relatively divisible) in $M$, and
		\item[(d)] $M = \bigcup _{n < \omega} M _n$.
	\end{enumerate}
	Then, there exists a countable, ascending chain
	\begin{equation}
		\label {Eq:CountChain-F}
		0 = F _0 < F _1 < \ldots < F _n < \ldots \quad (n < \omega)
	\end{equation}
	of modules, such that
	\begin{enumerate}
		\item[(A)] every $F _n$ is a completely decomposable module,
		\item[(B)] every $F _n$ has countable rank,
		\item[(C)] every $F _n$ is a pure (resp.,  relatively divisible) submodule of $F _{n + 1}$, and
		\item[(D)] every $F _n$ contains $M _n$ as a direct summand, say, $F _n = M _n \oplus K _n$. 
	\end{enumerate}
	Moreover, $M$ is a direct summand of $F = \bigcup _{n < \omega} F _n$.
\end{lemma}

\begin{proof}
	Suppose that we have already constructed the modules $F _0 , \ldots , F _n$ as desired, for some $n < \omega$. The module $M _{n + 1} \oplus K _n$ is a countable-rank, balanced-projective module, then there exists a countable-rank, completely decomposable module $F _{n + 1}$ containing $M _{n + 1} \oplus K _n$ as a direct summand. The module $F _n$ is a pure (resp.,  relatively divisible) submodule of $F _{n + 1}$, and the chain is thus constructed by induction. Moreover, the way that the links of \eqref {Eq:CountChain-F} were constructed guarantees that the sequence $\{ K _n \} _{n < \omega}$ forms an ascending chain under inclusion. Consequently, $F$ is the direct sum of $M$ and the module $\bigcup _{n < \omega} K _n$.
\end{proof}

\section{$G (\kappa) ^\ast$-families}

Let $\kappa$ be an infinite cardinal number. A \emph {$G (\kappa) ^\ast$-family} of a module $M$ is a set $\mathcal {B}$ of submodules of $M$ with the following properties:
\begin{enumerate}
	\item[(i)] $0 , M \in \mathcal {B}$,
	\item[(ii)] $\mathcal {B}$ is closed under unions of ascending chains, and
	\item[(iii)] for every $H \subset M$ of cardinality at most $\kappa$ and every $A _0 \in \mathcal {B}$, there exists $A \in \mathcal {B}$ containing both $A _0$ and $H$, such that $A / A _0$ has rank at most $\kappa$.
\end{enumerate}
Evidently, if a module $M$ can be decomposed as $M = \oplus _{\alpha \in \Omega} A _\alpha$, where each $A _\alpha$ has countable rank, then the set of all modules of the form $\oplus _{\alpha \in \Lambda} A _\alpha$, with $\Lambda \subset \Omega$, is a $G (\aleph _0) ^\ast$-family of submodules of $M$, called the \emph {family of partial direct sums} in the given decomposition of $M$.

\begin{assumption} 
	For the remainder of the present section, we assume that $M$ is a torsion-free module over an integral domain, for which there exists a continuous, well-ordered, ascending chain
	\begin{equation}
		\label{Eq:kappa-Chain}
		0 = M _0 < M _1 < \ldots < M _\nu < \ldots \quad (\nu < \kappa)
	\end{equation}
	of submodules of $M$, satisfying the following properties:
	\begin{enumerate}
		\item[$\mathbf {P _1}$] every $M _\nu$ is pure in $M$,
		\item[$\mathbf {P _2}$] every $M _\nu$ has a $G (\kappa) ^\ast$-family $\mathcal {B} _\nu$ consisting of pure submodules, and
		\item[$\mathbf {P _3}$] $M = \bigcup _{\nu < \kappa} M _\nu$.
	\end{enumerate}
\end{assumption}

The next results are reproduced from \cite {Fuchs-Salce2} (refer to Section XVI.1 for the proofs), and they are important tools in our investigation.

\begin{lemma}
	\label {Lemma:3-9}
	There exists a continuous, well-ordered, ascending chain 
	\begin{equation}
		\label{Eq:tau-Chain}
		0 = A _0 < A _1 < \ldots < A _\alpha < \ldots \quad (\alpha < \tau)
	\end{equation}
	of submodules of $M$, such that
	\begin{enumerate}
		\item[(a)] for every $\alpha < \tau$ and $\nu < \kappa$, $A _\alpha \cap M _\nu \in \mathcal {B} _\nu$,
		\item[(b)] for every $\alpha < \tau$, the factor module $A _{\alpha + 1} / A _\alpha$ has rank at most $\kappa$,
		\item[(c)] for every $\alpha < \tau$ and $\nu < \kappa$, $(A _\alpha \cap M _{\nu + 1}) + (A _{\alpha + 1} \cap M _\nu) \in \mathcal {B} _{\nu + 1}$, and
		\item[(d)] $M = \bigcup _{\alpha < \tau} A _\alpha$. \qed
	\end{enumerate}
\end{lemma}

Following the notation in Lemma \ref {Lemma:3-9}, the well-ordered collection of modules $A _\alpha + (A _{\alpha + 1} \cap M _\nu)$ under inclusion, with $\alpha < \tau$ and $\nu < \kappa$, are indeed the modules $N _\rho$ in the next general result.

\begin{theorem}
	\label{Lemma:3-10}
	There exists a continuous, well-ordered, ascending chain
\begin{equation*}
	0 = N _0 < N _1 < \ldots < N _\rho < \ldots \quad (\rho < \sigma)
\end{equation*}
of submodules of $M$, such that
\begin{enumerate}
	\item[(a)] every factor module $N _{\rho + 1} / N _\rho$ has rank at most $\kappa$,
	\item[(b)] every factor module $N _{\rho + 1} / N _\rho$ is isomorphic to a factor module $A / B$, with both $A$ and $B$ in some family $\mathcal {B} _\nu$ and $B < A$, and
	\item[(c)] $M = \bigcup _{\rho < \sigma} N _\rho$. \qed
\end{enumerate}
\end{theorem}

\section{Main results}

The results of the present section are mainly criteria for the balanced projectivity of modules, but they are valid also if the condition of balanced projectivity is replaced by the one of complete decomposability. We provide the proofs for the former case only, but the proofs of the latter are established \emph {verbatim}.

\begin{theorem}
	\label {Lemma:3}
	The torsion-free module $A$ is balanced-projective if there exists a continuous, well-ordered, ascending chain \eqref {Eq:tau-Chain} of submodules of $A$, such that
	\begin{enumerate}
		\item[(a)] every $A _\alpha$ is a balanced submodule of $A _{\alpha + 1}$,
		\item[(b)] every factor module $A _{\alpha + 1} / A _\alpha$ is balanced-projective, and
		\item[(c)] $A = \bigcup _{\alpha < \tau} A _\alpha$.
	\end{enumerate}
\end{theorem}

\begin{proof}
	The sequence $0 \rightarrow A _\alpha \rightarrow A _{\alpha + 1} \rightarrow A _{\alpha + 1} / A _\alpha \rightarrow 0$ is balanced-exact for every $\alpha < \kappa$, so that it splits. Then, there exists a balanced-projective submodule $B _\alpha$ of $A _{\alpha + 1}$, such that $A _{\alpha + 1} = A _\alpha \oplus B _\alpha$. The module $A$ is the direct sum of the modules $B _\alpha$, so balanced-projective.
\end{proof}

The following result is the version for balanced-projective modules, of the well-known Shelah-Eklof criterion of freeness from abelian group theory. For the definitions of the set-theoretic concepts and results employed in this theorem, we refer to \cite {Jech}.

\begin{theorem}
	Let $\kappa$ be an uncountable, regular ordinal number. A torsion-free module $M$ is balanced-projective if and only if there exists a $\kappa$-filtration \eqref {Eq:kappa-Chain} consisting of balanced-projective, balanced submodules of $M$, such that 
	\begin{equation*}
		E = \{ \alpha < \kappa : \exists \beta > \alpha \text{ such that } M _\beta / M _\alpha \text { is not balanced-projective} \}
	\end{equation*}
	is not stationary in $\kappa$.
\end{theorem}

\begin{proof}
	If $M$ is balanced-projective, then it is a direct summand of a completely decomposable module. By the rank version of Kaplansky's theorem \cite {Kaplansky}, $M$ is the direct sum of balanced-projective submodules of countable rank, say, $M = \oplus _{\alpha < \kappa} H _\alpha$. Let $B _\alpha = \oplus _{\beta < \alpha} H _\beta$, for every $\alpha < \kappa$. The set $C$ of indexes $\alpha < \kappa$ of those modules $M _\alpha$ which also appear in the filtration $\{ B _\beta \} _{\beta < \kappa}$ is a closed and unbounded set in $\kappa$. So, $\{ M _\alpha \} _{\alpha \in C}$ is a filtration of $M$ with the property that, for every $\alpha , \beta \in C$ with $\alpha < \beta$, $M _\beta / M _\alpha$ is balanced-projective. Therefore, $C$ does not intersect $E$ and, consequently, $E$ is not stationary in $\kappa$.
	
	Conversely, there is a closed and unbounded set $C$ in $\kappa$ which does not intersect $E$. Then $\{ M _\alpha \} _{\alpha \in C}$ is a filtration of balanced submodules of $M$, which can be rewritten as $\{ M _\alpha \} _{\alpha < \kappa}$ by relabeling. In this filtration, all the factor modules $M _{\alpha + 1} / M _\alpha$ are balanced-projective, and the conclusion follows from Theorem \ref {Lemma:3}.
\end{proof}

Finally, we give a generalization of Hill's criterion of freeness to balanced projectivity of modules over domains. Its reach clearly surpasses the criteria for complete decomposability proved by Rangaswamy in \cite {Rangas}.

\begin{theorem}
	\label{Thm:Main}
	A torsion-free module $M$ over a domain is balanced-projective if there exists a countable, ascending chain \eqref {Eq:CountChain} of submodules of $M$, such that
	\begin{enumerate}
		\item[(a)] every $M _n$ is balanced-projective,
		\item[(b)] every $M _n$ is pure in $M$, and
		\item[(c)] $M = \bigcup _{n < \omega} M _n$.
	\end{enumerate}
\end{theorem}

\begin{proof}
In view of Lemma \ref {Lemma:6}, we reduce the proof to the case when every $M _n$ is a completely decomposable submodule of $M$. So, for every $n < \omega$, fix a decomposition of $M _n$ into rank-one, torsion-free submodules, and let $\mathcal {B} _n$ be the $G (\aleph _0) ^\ast$-family of all partial direct sums in such decomposition. By Lemma \ref {Lemma:3-9} and Theorem \ref {Lemma:3-10}, there exists a continuous, well-ordered, ascending chain \eqref {Eq:tau-Chain} of submodules of $M$ whose union is $M$ itself, such that for every $\alpha < \tau$ and every $n < \omega$,
\begin{enumerate}
	\item[(i)] $A _\alpha \cap M _n \in \mathcal {B} _n$,
	\item[(ii)] $(A _\alpha \cap M _{n + 1}) + (A _{\alpha + 1} \cap M _n) \in \mathcal {B} _{n + 1}$, and
	\item[(iii)] the factor module
		\begin{equation*}
			\frac {A _\alpha + (A _{\alpha + 1} \cap M _{n + 1})} {A _\alpha + (A _{\alpha + 1} \cap M _n)} \cong \frac {A _{\alpha + 1} \cap M _{n + 1}} {(A _\alpha \cap M _{n + 1}) + (A _{\alpha + 1} \cap M _n)}
		\end{equation*}
		is completely decomposable.
\end{enumerate}
Since $(A _\alpha \cap M _{n + 1}) + (A _{\alpha + 1} \cap M _n)$ is a balanced submodule of $A _{\alpha + 1} \cap M _{n + 1}$, we have the following commutative diagram with balanced-exact first column and exact second column:
\begin{equation*}
	\xymatrix{
	 & 0 \ar[d] & 0 \ar[d] \\
	0 \ar[r] & (A _\alpha \cap M _{n + 1}) + (A _{\alpha + 1} \cap M _n) \ar[d]\ar[r]^\iota & A _\alpha + (A _{\alpha + 1} \cap M _n) \ar[d] \\
	0 \ar[r] & A _{\alpha + 1} \cap M _{n + 1} \ar[d]\ar[r]^{\iota ^\prime} & A _\alpha + (A _{\alpha + 1} \cap M _{n + 1}) \ar[d] \\
	 & \displaystyle {\frac {A _{\alpha + 1} \cap M _{n + 1}} {(A _\alpha \cap M _{n + 1}) + (A _{\alpha + 1} \cap M _n)}} \ar[d]\ar[r]^\cong & \displaystyle {\frac {A _\alpha + (A _{\alpha + 1} \cap M _{n + 1})} {A _\alpha + (A _{\alpha + 1} \cap M _n)}} \ar[d] \\
	 & 0 & 0}
\end{equation*}
Here, $\iota$ and $\iota ^\prime$ represent the inclusion maps. Lemma \ref {Lemma:2} implies that the second column is balanced-exact. In this way, the chain $\{ N _\rho \} _{\rho < \sigma}$ defined in the paragraph preceding Theorem \ref {Lemma:3-10}, is actually a continuous, well-ordered, ascending chain of submodules of $M$ satisfying the hypotheses of Theorem \ref {Lemma:3} for the module $M$. We conclude that $M$ is completely decomposable.
\end{proof}

The \emph {balanced-projective dimension} of a module $M$ is defined in similar way as the projective dimension, but taking balanced-projective resolutions instead of projective ones, and it is usually denoted by $\bpd M$ (refer to \cite {Fuchs-Salce2}). The following result ---which is a generalization of Auslander's lemma \cite {Auslander}--- is a direct consequence of Theorem \ref {Thm:Main} and induction.

\begin{corollary}
	A torsion-free module $M$ satisfies $\bpd M \leq k$ if there exists a countable, ascending chain \eqref {Eq:CountChain} of submodules of $M$, such that
\begin{enumerate}
	\item[(a)] every $M _n$ is balanced-projective,
	\item[(b)] every $M _n$ satisfies $\bpd M _n \leq k$, and
	\item[(c)] $M = \bigcup _{n < \omega} M _n$. \qed
\end{enumerate}
\end{corollary}

\end{document}